\documentclass{article}
\usepackage{geometry}                
\usepackage{mathtools}
\usepackage{amssymb}
\usepackage{amsmath}
\usepackage{amsthm,dsfont,yfonts}
\usepackage{rotating}
\usepackage{graphicx,pspicture}
\usepackage{array}
\usepackage{calc}
\usepackage{soul}
\usepackage[all]{xy}
\usepackage{amsthm}
\usepackage{lmodern}
\usepackage{color}
\usepackage{pstricks,pst-node,pst-tree}
\usepackage{graphics,graphicx}
\usepackage[british]{babel}
\usepackage{fancyhdr}
\usepackage{color}
\usepackage{pinlabel}
\usepackage{enumerate}

\newtheorem{theorem}{Theorem}[section]

\newtheorem{lemma}[theorem]{Lemma}
\newtheorem{prop}[theorem]{Proposition}
\newtheorem{remark}[theorem]{Remark}
\theoremstyle{definition}
\newtheorem{definition}[theorem]{Definition}

\newcommand{\Aut}{\operatorname{Aut}}
\newcommand{\spec}{\operatorname{spec}}

\newcommand{\Alt}{\operatorname{Alt}}
\newcommand{\diam}{\operatorname{diam}}
\newcommand{\St}{\operatorname{St}}
\newcommand{\rst}{\operatorname{rst}}

\newcommand{\h}{\hspace{2mm}}  

\usepackage{makeidx}
\makeindex

\usepackage[noindentafter]{titlesec}
\title{On the finitely generated Hausdorff spectrum of spinal groups}
\date{\today}
\author{Elisabeth Fink}

\begin{document}

\selectlanguage{british}

\maketitle

\begin{abstract}
We study the finitely generated Hausdorff spectrum of spinal automorphism groups acting on rooted trees. Given
any $\alpha \in [0,1]$, we construct a branch group $G_\alpha$ such that $G_\alpha$ has a finitely generated subgroup
$H$ where $H$ has Hausdorff dimension $\alpha$ in $G$. Using results by Barnea, Shalev and
Klopsch we further deduce that the finitely generated Hausdorff
spectrum of this group $G_\alpha$ contains $\mathcal{L}_\alpha \cup ([0, 1] \cap \mathcal{L})$, where $\mathcal{L}$ is
a countable subset of $\mathbb{Q}$ and $\mathcal{L}_\alpha$ is a certain set of countably many irrational numbers in
the interval $[0,\alpha]$. This answers a question of Benjamin Klopsch \cite{klopsch}.
\end{abstract}

\section{Introduction}

Groups of automorphisms acting on rooted trees have been studied recently. Well known examples of such are the
Grigorchuk group and the Gupta-Sidki groups. In this paper we use the notion of a Hausdorff dimension, a fractal
dimension, to investigate the sizes of finitely generated subgroups in branch groups. Addressing the question of how
large such a branch group $G$ is within $\Aut(T)$, where $T$ is a rooted tree, Barnea and Shalev \cite{barnea_shalev}
have computed an explicit formula for the Hausdorff dimension. 

\medskip

Ab\'ert and Vir\'ag \cite{abert_virag} have shown that there exist finitely generated subgroups of $\Aut(T)$ with
arbitrary
Hausdorff dimension. In \cite{siegenthaler} Siegenthaler explicitly computed the Hausdorff dimension of level-transitive
spinal groups. Fernandez-Alcober and Zugadi-Reizabal \cite{zugadi_fernandez} give an explicit set of values for the
dimension of certain spinal groups. In his thesis \cite{klopsch} B. Klopsch has shown that branch groups have full
subgroup Hausdorff spectrum $[0,1]$. He leaves the question open whether the finitely generated Hausdorff spectrum can
be transcendental. Here we give for all $\alpha \in [0,1]$ an explicit example of a branch group $G_\alpha$ with a
finitely generated subgroup $H$, such that $\bar{H}$ has dimension $\alpha$ in $G_\alpha$.

\medskip

Further considerations yield that the finitely generated Hausdorff spectrum of this constructed group $G_\alpha$
contains $\mathcal{L}_\alpha \cup ([0, 1] \cap \mathcal{L})$, where $\mathcal{L}$ is a countable subset of $\mathbb{Q}$
and $\mathcal{L}_\alpha$ is a certain set of countably many irrational numbers in the interval $[0,\alpha]$. We do not
know whether the parameters of this construction can be chosen such that for all $\nu \in \mathcal{L}_\alpha \cup
\left([0,1] \cap \mathbb{Q}\right)$ there exists a finitely generated subgroup $H$ of Hausdorff dimension $\nu$ in $G$.

\medskip

We suspect that an alternative construction may give rise to groups whose Hausdorff spectrum is purely rational. In
either case, it is however clear that the spectrum can only contain countably many values, as there are only countably
many finitely generated subgroups of a finitely generated group.

\section{Fractal Dimensions in Branch Groups}

We give a quick introduction on Hausdorff dimensions and explain how they can be defined in profinite groups. We refer
the reader to Falconer \cite{falconer} for more information.

\medskip

Let $(X,d)$ be a metric space, let $Y \subset X$ and $\alpha, \rho \in \mathbb{R}^+$. Define

\[\mathcal{H}_\rho^\alpha(Y) = \inf \sum_i \left(\diam S_i\right)^\alpha,\] where $\left\{S_i\right\}_{i=0}^\infty$ is
a cover of $Y$ by sets of diameter at most $\rho$, and the infimum is taken over all such covers. Note that
$\mathcal{H}_\rho^\alpha(Y)$ is non-increasing with $\rho$, and so the limit \[\mathcal{H}^\alpha = \lim_{\rho
\rightarrow \infty} \mathcal{H}_\rho^\alpha(Y)\] exists. It can be verified that $\mathcal{H}^\alpha(Y)$ is an outer
measure on $X$, the \emph{$\alpha$-dimensional Hausdorff measure}.

\begin{lemma}
If $\mathcal{H}^\alpha(Y) < \infty$ and $\alpha < \alpha'$, then $\mathcal{H}^{\alpha'}(Y)=0$.
\end{lemma}

We can now define the \emph{Hausdorff dimension} of a set $Y \subset X$:

\[\dim_H(Y) = \sup\left\{\alpha | \mathcal{H}^\alpha(Y)=\infty\right\} = \inf\left\{\alpha |
\mathcal{H}^\alpha(Y)=0\right\}.\]

A \emph{filtration} of $G$ is a descending chain of open normal subgroups $G=G_0 \geq G_1 \geq \dots \geq G_n \geq
\dots$ which forms a base of the neighborhoods of the identity. For such a series we have $\bigcap_{n=0}^\infty G_n =
\{1\}$. Now, let $G$ be a profinite group, equipped with a filtration $G_n$. Define an invariant metric $d$ on $G$ by
\[d(x,y)=\inf\left\{|G/G_n|^{-1} | xy^{-1} \in G_n\right\}.\] With respect to these definitions, Barnea and Shalev
proved the following theorem:

\begin{theorem}
Let $G$ be a profinite group with a filtration $\left\{G_n\right\}_{n=0}^\infty$ and let $H \leq G$ be a closed
subgroup. Then \[\dim_G(H)= \liminf_{n \rightarrow \infty} \frac{\log|H/\left(H \cap G_n\right)}{\log|G/G_n|},\] where
the Hausdorff dimension is computed with respect to the metric associated with the filtration $\left\{G_n\right\}$.
\end{theorem}

\begin{remark}
The Hausdorff dimension of $H \leq G$ depends in general on the chosen filtration $\left\{G_n\right\}$. In 
\cite[Example 2.5]{barnea_shalev} the authors give an example.
\end{remark}

When we talk about the Hausdorff dimension of a subgroup $H$ in $G$, we will from now on mean the dimension of its
closure, $\bar{H}$ in $\bar{G}$, which denotes the profinite completion of $G$ (see \cite{ribes_zaleskii} for a
definition).

\medskip

The \emph{Hausdorff spectrum} $\spec_H(G)$ of a group $G$ is the set of all values $\alpha \in [0,1]$ for which
there exists a subgroup $H$ such that $dim_G(H) = \alpha$:

\[\spec_H(G)=\left\{\dim_G(\bar{H}) | \h H \leq G\right\}.\] The \emph{finitely generated Hausdorff spectrum} of a
group $G$ is defined as

\[\spec_H^{fg}(G) = \left\{ \dim_G(\bar(H)) | \h H \leq G, H \mbox{ is finitely generated}\right\}.\]

\section{Rooted Trees and Automorphisms}

For the general concept of groups acting on rooted trees we refer to \cite{bartholdi}. In contrast to examples most
widely studied the rooted trees here can also be irregular in the sense that the valency of vertices on different
levels of the tree does not have to be the same. The tree, however, will still be spherically homogenous. 

\medskip

In this section we will recall some of the notation and definitions from \cite{bartholdi} and \cite{segal_fifg}.

\subsection{Trees}

A \emph{tree} is a connected graph which has no non-trivial cycles. If $T$ has a distinguished \emph{root} vertex $r$
it is called a \emph{rooted tree}. The distance of a vertex $v$ from the root is given by the length of the path from
$r$ to $v$ and called the \emph{norm} of $v$. The number \[d_v = | \{e \in E(T): e=\left(v_1, v_2\right), v = v_1
\textnormal{ or } v=v_2\}|\] is called the \emph{degree} of $v \in V(T)$. The tree is called \emph{spherically
homogeneous} if vertices of the same norm have the same degree. Let $\Omega(n)$ denote the set of vertices of distance
$n$ from the root. This set is called the $n$-th level of $T$. A spherically homogeneous tree $T$ is determined by,
depending on the tree, a finite or infinite sequence $\bar{l}=\left\{l_n\right\}_{n=1}$ where $l_n+1$ is the degree of
the vertices on level $n$ for $n \geq 1$. The root has degree $l_0$. Hence each level $\Omega(n)$ has $\prod_{i=0}^{n-1}
l_i$ vertices. Let us denote this number by $m_n = |\Omega(n)|$. We denote such a tree by $T_{\bar{l}}$. A tree is
called \emph{regular} if $l_i = l_{i+1}$ for all $i \in \mathbb{N}$. Let $T[n]$ denote the finite tree where all
vertices have norm less or equal to $n$ and write $T_v$ for the subtree of $T$ with root $v$.
For all vertices $v,u \in \Omega(n)$ we have that $T_u \simeq T_v$. Denote a tree isomorphic to $T_v$ for $v \in
\Omega(n)$ by $T_n$. This will be the tree with defining sequence $\left(l_n, l_{n+1}, \dots \right)$. To each sequence
$\bar{l}$ we associate a sequence $\left\{X_n\right\}_{n \in \mathbb{N}}$ of alphabets where $X_n=\left\{v_1^{(n)},
\dots, v_{l_n}^{(n)}\right\}$ is an $l_n$-tuple so that $|X_n|=l_n$.  A path beginning at the root of length $n$ in
$T_{\bar{l}}$ is identified with the sequence ${x_1,\dots, x_i, \dots, x_n}$ where $x_i \in X_i$ and infinite paths are
identified in a natural way with infinite sequences. Vertices will be identified with finite strings in the alphabets
$X_i$. Vertices on level $n$ can be written as elements of $Y_n  = X_0 \times \dots \times X_{n-1}$. Alphabets induce
the lexicographic order on the paths of a tree and therefore the vertices.

\subsection{Automorphisms}

An \emph{automorphism} of a rooted tree $T$ is a bijection from $V(T)$ to $V(T)$ that preserves edge incidence and the
distinguished root vertex $r$. The set of all such bijections is denoted by $\Aut T$. This group induces an imprimitive
permutation on $\Omega(n)$ for each $n \geq 2$. Consider an element $g \in \Aut(T)$.  Let $y$ be a letter from $Y_n$,
hence a vertex of $T[n]$ and $z$ a vertex of $T_n$. Then $g(y)$ induces a vertex permutation $g_y$ of $Y_n$. If we
denote
the image of $z$ under $g_y$ by $g_y(z)$ then \[g(yz)= g(y)
g_y(z).\]

\medskip

With any group $G \leq \Aut T$ we associate the subgroups \[\St_G(u)=\left\{g \in G: u^g=u\right\},\] the
\emph{stabilizer} of a vertex $u$. Then the subgroup \[\St_G(n)=\bigcap_{u \in \Omega(n)} \St_G(u)\] is called the
\emph{$n$-th level stabilizer} and it fixes all vertices on the $n$-th level. Another important class of subgroups
associated with $G \leq \Aut T$ consists of the \emph{rigid vertex stabilizers} \[\rst_G(u)=\left\{g \in G: \forall v
\in
V(T)
\setminus V(T_u): v^g=v\right\}.\]
The subgroup \[\rst_G(n)= \rst_G(u_1) \times \dots \times \rst_G(u_{m_n})\] is called the \emph{$n$-th level rigid
stabilizer}. Obviously $\rst_G(n) \leq \St_G(n)$.

\medskip

The \emph{support} of an automorphism $g$ is the set of all vertices that $g$ acts non-trivially on. If the support of
an automorphism $g \in G$ only containes the root, then we call $g$ a \emph{rooted automorphism}. We choose an infinite
path $P=\left(p_n\right)_{n \geq 0}$, starting at the root. Following the definition in \cite{zugadi_fernandez}, if we
consider, for every $n \geq 1$, and immediate descendant $s_n$ of $p_{n-1}$ not lying in $P$, we say that the sequence
$S=\left(s_n\right)_{n \geq 1}$ is a \emph{spine} of the tree $T$. An element $g \in G$ is a \emph{spinal automorphism}
if the support of $g$ is contained in $S$.

\begin{definition}
A \emph{spinal group} $G$ acting on a rooted tree $T$ is a subgroup of $\Aut(T)$ which is generated only by a set $A$
of rooted automorphisms and a set $B$ of spinal automorphisms.
\end{definition}

\section{The Construction of $G$}\label{sec_group}

In this Section we explain the construction of the group $G$ with the desired properties. We will then show in the next
section that $G$ indeed has those properties.

\medskip

Denote by $A_k$ the alternating group acting on the set $\left\{1, \dots, k\right\}$. Every group $A_k$ is generated by
a $3$-cycle and an $k$-cycle (\cite{dixon_mortimer}): 
\[\tau_k = ((k-2)(k-1)k), \h \sigma_k = (1 \dots k).\]

Let $\left\{l_i\right\}_{n \geq 0}$ be a sequence of natural numbers and let $\{A_{l_i}\}_{i \in \mathbb{N}}$ be a
sequence of alternating groups acting on the sets $\left\{1, \dots, l_i\right\}$. We study the group
\[G=\left<\tau_{l_0}, \sigma_{l_0}, \zeta, \psi \right> \h \h\] where $\zeta$ and $\psi$ are
recursively defined on each level $n$ by
\[\zeta_n=\left(\zeta_{n+1},\tau_{l_{n+1}},1,\ldots,1\right)_n,\]\[\psi_n=\left(\psi_{n+1},\sigma_{l_{n+1}},1,\ldots,
1\right)_n.\] This means the action on the first vertex of level $n$ is given by $\zeta_{n+1}$ or $\psi_{n+1}$ and the
action on the second vertex by the rooted automorphism $\tau_{n+1}$ or $\sigma_{n+1}$. Figure \ref{figure_zeta} depicts
the action of the automorphism $\zeta$ and $\psi$ on the tree. The action of $\zeta$ and $\psi$ on all unlabelled
vertices $v$ in the Figure will be given by the identity on $T_u$.

\begin{figure}[ht!]
\centering

\labellist
\pinlabel \LARGE{$\zeta$} at 50 200
\pinlabel $\tau_{l_1}$ at 128 180
\pinlabel $\zeta_1$ at 90 183

\pinlabel $\tau_{l_2}$ at 95 140
\pinlabel $\zeta_2$ at 62 143

\pinlabel $\tau_{l_3}$ at 66 95
\pinlabel $\zeta_3$ at 35 100

\pinlabel $\tau_{l_4}$ at 40 55
\pinlabel $\zeta_4$ at 8 60

\pinlabel $\zeta_0$ at 135 215

\pinlabel \LARGE{$\psi$} at 270 200
\pinlabel $\sigma_{l_1}$ at 350 180
\pinlabel $\psi_1$ at 310 183

\pinlabel $\sigma_{l_2}$ at 320 140
\pinlabel $\psi_2$ at 280 143

\pinlabel $\sigma_{l_3}$ at 288 95
\pinlabel $\psi_3$ at 250 100

\pinlabel $\sigma_{l_4}$ at 258 55
\pinlabel $\psi_4$ at 225 60

\pinlabel $\psi_0$ at 355 215

\endlabellist
\centerline{
\includegraphics[scale=1]{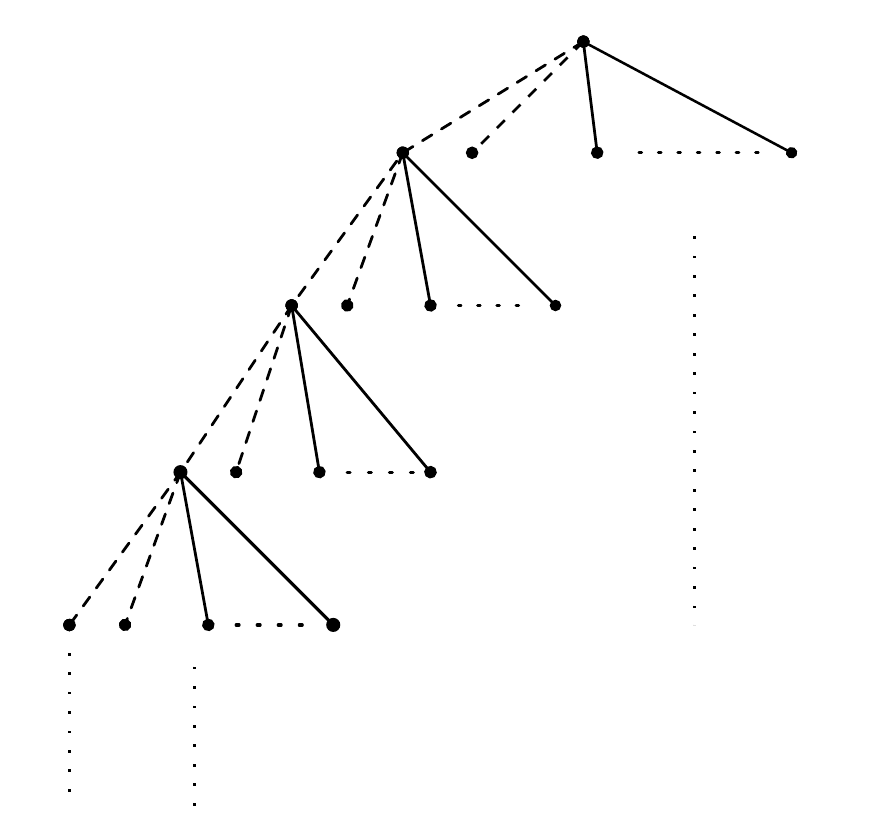}
\includegraphics[scale=0.93]{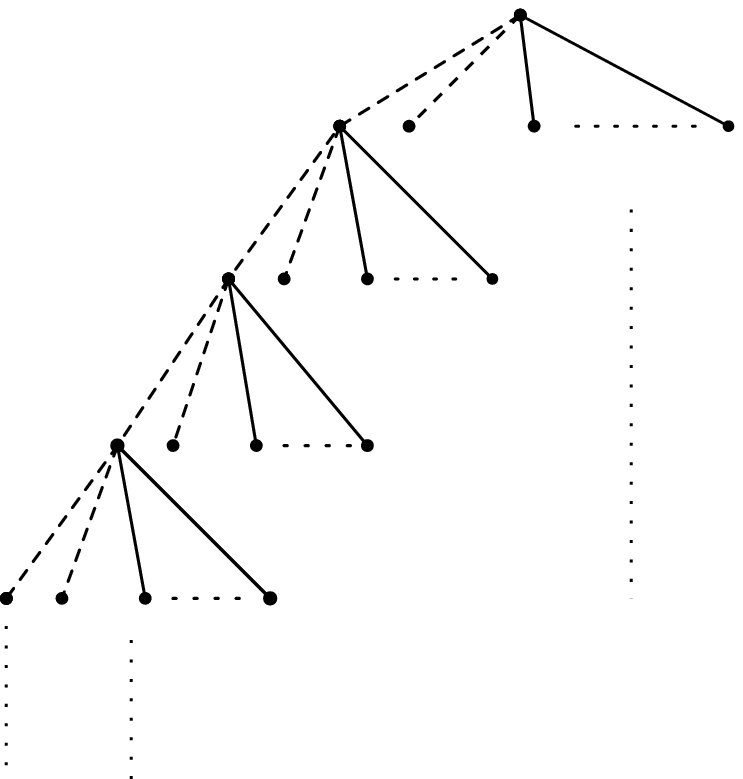} 
}
\caption{Portrait of the automorphisms $\zeta$ and $\psi$.}\label{figure_zeta}
\end{figure}

The following Proposition is a simple inductive argument using that the alternating groups are perfect and is proved in
\cite{segal_subgroupGrowth}:

\begin{prop}\label{prop_actionLevel}
$G$ acts as the iterated wreath product $A_{l_{n-1}} \wr \dots \wr A_{l_1} \wr A_{l_0}$ on the set $\Omega(n)$ of $m_n$
vertices of each level $n$.
\end{prop}

\section{The Finitely Generated Spectrum}

In this section we show that for every $\alpha \in [0,1]$ we can construct a group $G_\alpha$ as in Section
\ref{sec_group} such that there exists a finitely generated subgroup $H \leq G_\alpha$ with $\dim_G(H)=\alpha$. We fix
for the rest of this paper the filtration $G_i = \St_G(i)$ of $G$.

\medskip

First a simple Lemma on the approximation of a number in the interval $[0,1]$.

\begin{lemma}\label{lemma_alpha}
For every $\alpha \in (0,1)$ there exists a sequence $\left\{l_i\right\}, l_i \in \mathbb{N},$ of integers $l_i \geq 5$
such that \[\lim_{i \rightarrow \infty} \prod_{j=0}^i \frac{l_i-2}{l_i} = \alpha.\]
\end{lemma}

\begin{proof}
Choose $l_0$ such that 
\[\frac17 (6+\alpha) > \frac{l_0-2}{l_0} > \alpha.\] Further choose $l_i$ for $i \geq 1$ such that
\[\frac17\left(\frac{6 \cdot \alpha}{\prod_{j=0}^{i-1} \frac{l_j-2}{l_j}} + \alpha\right) >  \frac{l_i-2}{l_i} >
\frac{\alpha}{\prod_{j=0}^{i-1} \frac{l_j-2}{l_j}}.\] We can without loss of generality assume that $l_i \geq 5$.
\end{proof}

\begin{remark}\label{remark_primes}
The approximating sequence $\left\{l_i\right\}_{i \geq 0}$ is not unique. It will prove useful to choose each $l_i$
such, that it has many different prime factors.
\end{remark}

We can now show that we can construct a finitely generated subgroup $H$ of dimension $\alpha$ in $G$ if we choose the
defining sequence $\left\{l_i\right\}$ for $G$ depending on $\alpha$ in the right way.

\begin{theorem}\label{thm_alpha}
For every $\alpha \in [0,1]$ there exists a branch group $G_\alpha$ and a finitely generated subgroup $H \leq G_\alpha$
such that $\dim_{G_\alpha}(H)=\alpha$. Further, $H$ is again a finitely generated branch group.
\end{theorem}

\begin{proof}
If $\alpha=1$ set $H=G$ and if $\alpha=0$ set $H=1$. Otherwise let us choose the sequence $\left\{l_i\right\}$ as in
Lemma \ref{lemma_alpha} such that \[\lim_{i \rightarrow \infty} \prod_{j=0}^i \frac{l_i-2}{l_i} = \alpha.\] Let
$G=\left<\tau_0, \psi_0, \zeta, \psi\right>$ be as described in Section \ref{sec_group}. The elements
$\kappa_n, \rho_n \in \Alt(n)$ with \[\kappa_n = \sigma_n^{-2} \tau_n \sigma_n^2 =((n-4)(n-3)(n-2)) \quad \mbox{and}
\quad \rho_n = \tau^2 \sigma = (1\dots (n-2))\]
generate the subgroup $\Alt(n-2) \leq \Alt(n)$. We use this to construct subgroups acting on $l_i-2$ points of order
$\left(l_i-2\right)!$ on each level. This gives us $\Alt\left(l_i-2)\right) \leq \Alt\left(l_i\right)$ and we prove that
the closure of the spinal subgroup

\[H = \left<\kappa_0,\rho_0, \xi, \theta \right>\] with $\xi_n = \left(\xi_{n+1}, \kappa_{l_{n+1}}, 1, \dots, 1\right)$
and
$\theta_n = \left(\theta_{n+1}, \rho_{l_{n+1}}, 1, \dots, 1\right)$ has dimension \[\alpha = \dim_G(H) =
\lim_{i\rightarrow
\infty}
\prod_{k=0}^i \frac{l_k-2}{l_k}\] in $\bar{G} = \lim_{\infty \leftarrow i} \Alt\left(l_i\right) \wr \dots \wr
\Alt\left(l_0\right)$. The subgroup $H$ is obviously finitely generated. We saw above that $\kappa_{l_0} =
\tau_{l_0}^{\sigma_{l_0}^2}$ and $\rho_{l_0} = \tau_{l_0}^2 \sigma_{l_0}$. It follows that $\xi = \psi^{-2} \zeta
\psi^2$ and
$\xi = \zeta^2 \psi$. We obtain from Proposition \ref{prop_actionLevel} that \[G/\St_G(n) = A_{l_{n-1}} \wr \dots \wr
A_{l_0}.\] Further it is easy to see that then \[H/\left(\St_G(n) \cap H\right) = H/\St_H(n) =
A_{\left(l_{n-1}-2\right)} \wr \dots \wr A_{\left(l_0-2\right)}.\] The formula for the dimension $\dim_G(H)$ of the
closure of $H$ in $\bar{G}$ is hence given by

\begin{equation}\label{eq_limitL}
L = \lim_{i \rightarrow \infty} \frac{\log\left(\left(\frac{\left(l_i-2\right)!}{2}\right)^{\prod_{j=0}^{i-1} l_j-2}
\cdot \dots
\cdot \left(\frac{\left(l_0-2\right)!}{2}\right)\right)}{\log\left(\left(\frac{l_i!}{2}\right)^{\prod_{j=0}^{i-1} l_j}
\cdot \dots \cdot \frac{l_0!}{2}\right).}
\end{equation}

We separate this into

\[\frac{-\log\left(2^{\sum_{j=0}^{i-1} \prod_{k=0}^{j} \left(l_k-2\right)}\right) + \log \left(
\left(l_i-2\right)!^{\prod_{k=0}^{i-1} \left(l_k-2\right)} \cdot \dots \cdot
\left(l_0-2\right)!\right)}{-\log\left(2^{\sum_{j=0}^{i-1} \prod_{k=0}^j l_j} \right) +
\log\left(l_i!^{\prod_{j=0}^{i-1}l_j} \cdot \dots \cdot l_0!\right)}.\]

Let us denote this fraction to be of the form $\frac{-\log A + \log B}{-\log C + \log D}$. This can be computed
separately as

\[-\frac{\log A}{-\log C + \log D} + \frac{\log B}{-\log C + \log D} = - \frac{1}{-\frac{\log A}{\log C} + \frac{\log
D}{\log A}} + \frac{1}{-\frac{\log B}{\log C} + \frac{\log D}{\log B}}.\]

Simple estimations yield that 

\begin{enumerate}
 \item $\lim \frac{-\log A}{\log C} = 0$,
 \item $\lim \frac{\log D}{\log A} = \infty$ and
 \item $\lim \frac{-\log B}{\log C} = 0$.
\end{enumerate}

Hence we concentrate on $\lim \left(\frac{\log D}{\log B}\right)^{-1}$ by computing $\lim \frac{\log B}{\log D}$ which
can be written as

\[L=\lim_{i\rightarrow \infty} \frac{\prod_{j=0}^{i-1} \left(l_j-2\right)
\log\left(\left(l_i-2\right)! \cdot \dots \cdot \left(l_0-2\right)!^{\prod_{j=0}^{i-1}
\frac{1}{l_j-2}}\right)}{\prod_{j=0}^{i-1} l_j \log\left(l_i! \cdot \dots \cdot l_0!^{\prod_{j=0}^{i-1} \frac{1}{l_j
}}\right)}.\]

If we set $\alpha_i = \prod_{j=0}^{i-1} \frac{l_j-2}{l_j}$ then this can be estimated with
\begin{equation}\label{eq_stirling}
e\left(\frac{n}{e}\right)^{n} \leq n! \leq e\cdot \left(\frac{n+1}{e}\right)^{n+1},\end{equation} a
consequence of Stirling's formula, as

\[L \leq \lim_{i\rightarrow \infty} \alpha_i \cdot \frac{\log\left(e^{i+1} \cdot \left(\frac{l_{i}}{e}\right)^{l_i}
\cdot
\dots \cdot
\left(\frac{l_0}{e}\right)^{\prod_{j=1}^{i-1} \frac{1}{l_j}}\right)}{\log \left(l_i! \cdot
\left(l_{i-1}!\right)^{\frac{1}{l_{i-1}}} \cdot \dots \cdot \left(l_0!\right)^{\prod_{j=0}^{i-1}
\frac{1}{l_j}} \right)}\]\[\leq 
\lim_{i\rightarrow \infty} \alpha_i \cdot \frac{\log\left(e^{i+1} \cdot \left(\frac{l_{i}}{e}\right)^{l_i} \cdot \dots
\cdot
\left(\frac{l_0}{e}\right)^{\prod_{j=1}^{i-1} \frac{1}{l_j}}\right)}{\log \left(e^{i+1} \cdot
\left(\frac{l_i}{e}\right)^{l_i} \cdot \dots \cdot \left(\frac{l_0}{e}\right)^{\prod_{j=1}^{i-1}
\frac{1}{l_j}} \right)}=\alpha.\]

For the other inequality we see from \eqref{eq_limitL} that

\[L \geq \alpha_i \cdot \frac{\log\left(\left(l_i-2\right)! \cdot \left(l_{i-1}-2\right)^{\frac{1}{l_i-2}} \cdot \dots
\cdot \left(l_0-2\right)!^{\prod_{j=0}^{i-1} \frac{1}{l_j}}\right) }{\log \left(l_i! \cdot l_{i-1}!^{\frac{1}{l_{i-1}}}
\cdot \dots \cdot l_0!^{\prod_{j=0}^{i-1} \frac{1}{l_j}}\right)}.\] We split up $k! = k \cdot \left(k-1\right)
\cdot \left(k-2\right)!$ for all terms in the denominator and write the logarithm as a sum:

\[\log \left(l_i! \cdot l_{i-1}!^{\frac{1}{l_{i-1}}} \cdot \dots \cdot l_0!^{\prod_{j=0}^{i-1} \frac{1}{l_j}}\right)=
\log \left(l_i \cdot l_{i-1}^{\frac{1}{l_{i-1}}} \cdot \dots \cdot l_0^{\prod_{j=0}^{i-1} \frac{1}{l_j}} \right)\]
\[+\log \left(\left(l_i-1\right) \cdot \left(l_{i-1}-1\right)^{\frac{1}{l_{i-1}}} \cdot \dots \cdot
\left(l_0-1\right)^{\prod_{j=0}^{i-1} \frac{1}{l_j}}\right)\]
\[+\log\left(\left(l_i-2\right)! \cdot \left(l_{i-1}-2\right)!^{\frac{1}{l_i-2}} \cdot \dots
\cdot \left(l_0-2\right)!^{\prod_{j=0}^{i-1} \frac{1}{l_j}}\right).\]

We divide all summands in the denominator by the nominator and get

\[L \geq \alpha_i \cdot \frac{1}{1 + T_1 + T_2 }\] where 

\begin{equation*}\label{eq_T1}
T_1 = \frac{\log \left(l_i \cdot
l_{i-1}^{\frac{1}{l_{i-1}}}\cdot \dots \cdot l_0^{\prod_{j=0}^{i-1} \frac{1}{l_j}}
\right)}
{\log\left(\left(l_i-2\right)! \cdot \left(l_{i-1}-2\right)!^{\frac{1}{l_i-2}} \cdot \dots
\cdot \left(l_0-2\right)!^{\prod_{j=0}^{i-1} \frac{1}{l_j}}\right)}
\end{equation*} and 
\begin{equation*}\label{eq_T2}
T_2 = \frac{\log \left(\left(l_i-1\right) \cdot
\left(l_{i-1}-1\right)^{\frac{1}{l_{i-1}}}\cdot \dots \cdot
\left(l_0-1\right)^{\prod_{j=0}^{i-1} \frac{1}{l_j}}\right)}
{\log\left(\left(l_i-2\right)! \cdot \left(l_{i-1}-2\right)!^{\frac{1}{l_i-2}} \cdot \dots
\cdot \left(l_0-2\right)!^{\prod_{j=0}^{i-1} \frac{1}{l_j}}\right)}.
\end{equation*}
Assuming $l_i \geq 5$ as stated in Lemma \ref{lemma_alpha} for all $i \geq 0$ we can estimate 

\begin{equation*}\label{eq_sum}
\sum_{j=0}^{n} \prod_{k=0}^j \frac{1}{l_k} \leq \sum_{k=1}^n \frac{1}{2^k} = 1
\end{equation*}

with which we obtain the inequality 
\begin{equation}\label{eq_l_i2}
\left(l_i \cdot l_{i-1}^{\frac{1}{l_{i-1}}}\cdot \dots \cdot
l_0^{\prod_{j=0}^{i-1} \frac{1}{l_j}} \right) \leq \left(l_i \cdot l_i^{\frac{1}{l_{i-1}}}\cdot \dots \cdot
l_i^{\prod_{j=0}^{i-1} \frac{1}{l_j}} \right) \leq l_i \cdot \left( l_i^{\frac{1}{l_{i-1}}}\cdot \dots \cdot
l_i^{\prod_{j=0}^{i-1} \frac{1}{l_j}} \right) \leq l_i^2.
\end{equation}

It is easy to see that $T_2 \leq T_1$ and so $T_1 + T_2 \leq 2 T_1$. We use the estimate \eqref{eq_l_i2} in the
nominator of $T_1$ and further \[\log\left(\left(l_i-2\right)! \cdot \left(l_{i-1}-2\right)!^{\frac{1}{l_i-2}} \cdot
\dots
\cdot \left(l_0-2\right)!^{\prod_{j=0}^{i-1} \frac{1}{l_j}}\right) \leq \log\left(\left(l_i-2\right)!\right).\]

The assumption $l_i \geq 5$ further gives that $l_i-2 \geq \frac{l_i}{2}$. Using \eqref{eq_stirling} again this
gives us 
\[T_1 \leq \frac{2 \log l_i}{\log \left(l_i -2\right)!} \leq \frac{2
\log l_i}{\log\left(e\left(\frac{l_i-2}{e}\right)^{l_i-2} \right)}\leq 2 \cdot \frac{\log l_i}{\left(l_i-2\right)
\log\left(l_i-2\right)} \leq \frac{8}{l_i} \longrightarrow \infty.\] Hence $L \geq \alpha$, and so $L = \alpha$. For the
last part, we observe that $H$ is a finitely generated branch group acting on the tree with defining sequence
$\left\{l_n-2\right\}_{n \geq 0}$.
\end{proof}

The proof of the Theorem \ref{thm_alpha} determines a sequence $\left\{l_i\right\}_{i \geq 0}$. We fix this sequence
for the rest of this document. 

\medskip

Following \cite{ribes_zaleskii}, we say a group $\Gamma$ is \emph{strongly complete} if
it satisfies any of the following conditions, which are easily seen to be equivalent:

\begin{enumerate}[(a)]
 \item Every subgroup of finite index in $\Gamma$ is open,
 \item $\Gamma$ is equal to its own profinite completion, 
 \item Every group homomorphism from $\Gamma$ to any profinite group is continuous.
\end{enumerate}

A powerful Theorem by Nikolov and Segal \cite[Theorem 1.1]{dan_nik} states

\begin{theorem}\label{thm_DanNik}
Every finitely generated profinite group is strongly complete.
\end{theorem}

One of the results in the paper on Hausdorff dimensions by Barnea and Shalev is the following:

\begin{lemma}\label{lemma_barneaShalev}
Let $G$ be a profinite group. If $H$ is an open subgroup of $G$, then $\dim_G(H)=1$ and if $H$ is a finite subgroup in
$G$, then $\dim_G(H)=0$.
\end{lemma}

Combining this with Theorem \ref{thm_DanNik}, we get

\begin{lemma}
Let $H$ be a subgroup of finite index in a finitely generated branch group $G$. Then $\dim_G(H)=1$.
\end{lemma}

\begin{proof}
Theorem \ref{thm_DanNik} asserts that $H$ is open in $G$. Hence we can apply Lemma \ref{lemma_barneaShalev} to conclude
that $\dim_G(H)=1$.
\end{proof}

\begin{lemma}\label{lemma_chainRule}
Let $G$ be a finitely generated branch group and $H$ and $K$ subgroups such that \[\dim_G(H)=\alpha, \h
\dim_H(K)=\beta.\] If we assume that $H$ is again a branch group, then $\dim_G(K)=\alpha \cdot \beta$.
\end{lemma}

\begin{proof}
 This follows straight from
\[dim_{G}(K) = \lim sup_{n \rightarrow \infty} \frac{\log(|K/\St_K(n)|)}{\log (|G/\St_{G}(n)|)} = \lim sup_{n
\rightarrow \infty} \frac{\log(|H/\St_H(n)|)}{\log (|G/\St_{G}(n)|)} \cdot \frac{\log(|K/\St_K(n)|)}{\log
(|H/\St_{H}(n)|)}\]
\[=\lim sup_{n \rightarrow \infty} \frac{\log(|H/\St_H(n)|)}{\log (|G/\St_{G}(n)|)} \cdot \lim sup_{n
\rightarrow \infty} \frac{\log(|K/\St_K(n)|)}{\log (|H/\St_{H}(n)|)}=\dim_G(H) \cdot \dim_H(K)\] because the limit of
both products exists by the assumptions $\dim_G(H)=\alpha$ and $\dim_H(K)=\beta$.
\end{proof}

Let $\mathcal{L}$ be the set of rationals

\begin{equation}\label{eq_L}
\mathcal{L} = \left\{q \h | \h q \in [0,1] \cap \mathbb{Q}, \exists \left\{j_1, \dots,j_{r_q}\right\}\subset
\mathbb{N} \h \mbox{with} \h q \cdot \prod_{k=1}^{r_q} \left(l_{j_k}-2\right) \in \mathbb{Z}\right\},\end{equation}

the set of all rational numbers $q \in [0,1] \cap \mathbb{Q}$ such that there exists a set $\left\{j_1, \dots,
j_{r_q}\right\} \subset \mathbb{N}$ with $q \cdot \prod_{k=1}^{r_q} \left(l_{j_k}-2\right) \in \mathbb{Z}$. 

\begin{prop}\label{prop_rationalDim}
Let $G$ be a finitely generated branch group. For every $\delta \in \mathcal{L}$ there exists a finitely generated
subgroup $H$ with $\dim_G(H)=\delta$.
\end{prop}

\begin{proof}
We follow a similar idea as Klopsch in his thesis (\cite{klopsch}), using the rigid level stabilizers of $G$. Those
have, by the hypothesis that $G$ is a branch group, finite index in $G$, hence are again finitely generated. The
subgroup $\rst_G(n)$ is the direct product $\rst_G(n) = \prod_{i=1}^{m_n} \rst_G(v)$ where $v$ is a vertex of level $n$.
It
follows straight from the notion of a Hausdorff dimension in branch groups that $H = \prod_{i=1}^{k}$ has dimension
$\dim_G(H)=\frac{k}{m_n}$ in $G$. The desired dimension $\delta$ can be written as $\delta = \frac p q = \frac{1}{m_n}
\cdot \frac{m_n a}{b}=\frac{\zeta \cdot a}{m_n}$ with $\zeta = \frac{m_n}{b}$ for every $n \geq 0$. By assumption there
exists $n_0$ such that $\zeta \in \mathbb{Z}$ for all $n \geq n_0$. Hence $\dim_G(H) = \delta$.
\end{proof}

We now see that a good choice of the sequence $\left\{l_i\right\}_{n \geq 0}$ allows the construction of a richer
spectrum as remarked in \ref{remark_primes}. Using Proposition \ref{prop_rationalDim} we can then obtain a more detailed
description of the finitely generated Hausdorff spectrum of $G$, using the definition of $\mathcal{L}$ from
\eqref{eq_L}.

\begin{theorem}\label{thm_spectrum}
For every $\alpha \in [0,1]$ there exists a branch group $G_\alpha$ such that \[\mathcal{L}_\alpha \cup
\mathcal{L} \subseteq \spec^{fg}_H(G),\] where $\mathcal{L}_\alpha =
\left\{l \cdot \alpha | \h l \in \mathcal{L}\right\}$.
\end{theorem}

\medskip

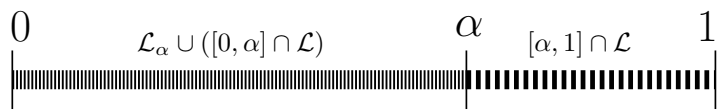
\begin{figure}[ht!]
\centering

\labellist
\pinlabel \LARGE{$\alpha$} at 250 45
\pinlabel \LARGE{$1$} at 380 45
\pinlabel \LARGE{$0$} at 5 45
\pinlabel {$\mathcal{L}_\alpha \cup \left([0,\alpha] \cap \mathcal{L}\right)$} at 120 35
\pinlabel {$[\alpha,1] \cap \mathcal{L}$} at 310 35
\endlabellist

\includegraphics[scale=0.7]{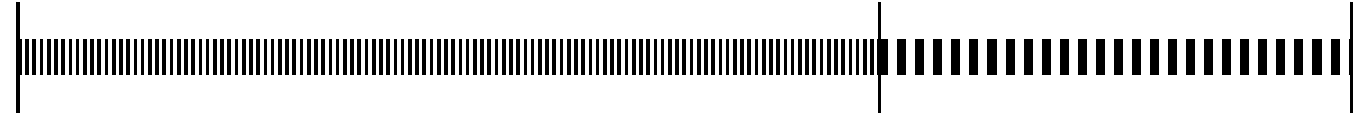}
\caption{Hausdorff spectrum of $G_\alpha$.}\label{figure_spectrum}
\end{figure}

\begin{proof}
If $\alpha \in \mathcal{L}$, then we apply Proposition \ref{prop_rationalDim}. Otherwise, Theorem \ref{thm_alpha}
yields that there exists a finitely generated subgroup $H$ with $\dim_{G_\alpha}(H)=\alpha$,
that is itself a branch group. Therefore by Proposition \ref{prop_rationalDim} there exists $K \leq H$ with
$\dim_H(K)=\delta$ for every $\delta \in \mathcal{L}$. Lemma \ref{lemma_chainRule} now asserts that
$\dim_{G_\alpha}(K)=\alpha \cdot \delta$.
\end{proof}

\medskip 
email: fink@maths.ox.ac.uk

\end{document}